\documentclass[12pt,reqno]{amsart}

\usepackage[usenames, dvipsnames]{color}
\usepackage{amsmath, amsthm, amssymb}
\usepackage{amsfonts}
\usepackage[ansinew]{inputenc}
\usepackage{graphicx}
\usepackage[english]{babel}
\usepackage{enumitem}

\usepackage[bookmarksnumbered,colorlinks, linkcolor=blue, citecolor=red, pagebackref, bookmarks, breaklinks]{hyperref}

\newtheorem{thm}{Theorem}[section]
\newtheorem{cor}[thm]{Corollary}
\newtheorem{lem}[thm]{Lemma}
\newtheorem{prop}[thm]{Proposition}

\newtheorem{rem}{Remark}

\numberwithin{equation}{section}

\newcommand{\ve}{\varepsilon}
\newcommand{\gfls}{(-\Delta_g)^s}
\newcommand{\R}{\mathbb{R}}

\newcommand{\N}{\mathbb{N}}
\newcommand{\F}{\mathcal{F}}
\newcommand{\J}{\mathcal{J}}
\newcommand{\tJ}{\widetilde{\mathcal{J}}}
\newcommand{\lam}{\lambda}
\def\d{\mathbf {d}}

\newcommand{\cde}{\stackrel{*}{\rightharpoonup}}
\newcommand{\cd}{\rightharpoonup}
\newcommand{\average}{{\mathchoice {\kern1ex\vcenter{\hrule height.4pt
width 6pt depth0pt} \kern-9.7pt} {\kern1ex\vcenter{\hrule
height.4pt width 4.3pt depth0pt} \kern-7pt} {} {} }}

\def\R{\mathbb{R}}

\begin{document}

\title[Fractional eigenvalues in Orlicz spaces with no $\Delta_2$ condition]
{Fractional eigenvalues in Orlicz spaces with no $\Delta_2$ condition}

\author{Ariel Salort}
\address{Instituto de Calculo (UBA - CONICET) and Departamento de Matematica, FCEyN, Universidad de Buenos Aires, Pabellon I, Ciudad Universitaria (1428), Buenos Aires, Argentina. }

\email{asalort@dm.uba.ar}

\author{Hern\'an Vivas}
\address{Centro Marplatense de Investigaciones Matem\'aticas (UNMDP-CIC)}
\address{Departamento de Matem\'atica, FCEYN, UNMDP \hfill\break \indent De\'an Funes 3350, 7600, Mar del Plata, Argentina}
\email{havivas@mdp.edu.ar}

\keywords{Orlicz spaces, Fractional partial differential equations, Nonlinear eigenvalue problems.}

\subjclass[2010]{46E30, 35R11, 35P30}

\begin{abstract}
We study the eigenvalue problem for the $g-$Laplacian operator in fractional order Orlicz-Sobolev spaces, where $g=G'$ and neither $G$ nor its conjugated function satisfy the $\Delta_2$ condition. Our main result is the existence of a nontrivial solution to such a problem; this is achieved by first showing that the corresponding minimization problem has a solution and then applying a generalized Lagrange multiplier theorem to get the existence of an eigenvalue. Further, we prove closedness of the spectrum and some properties of the eigenvalues and, as an application, we show existence for a class of nonlinear eigenvalue problems.
\end{abstract}

\maketitle
\tableofcontents

\section{Introduction and main results} \label{sec.intro}

The main goal of this article is to study the eigenvalue problem
\begin{equation}\label{eq.eigen}
\left\{ \begin{array}{cccc}
\gfls u  & = & \lambda g\left(|u|\right)\frac{u}{|u|} &  \textrm{ in }\Omega \\
u & = & 0  &\textrm{ in }\R^n\setminus\Omega,
\end{array} \right.
\end{equation}
where $\lambda\in\R$ and $\Omega$ is a bounded open subset of $\R^n$ with Lipschitz boundary. Here $\gfls$ is the fractional $g-$Laplacian defined in \cite{FBS}: 
\begin{equation}\label{eq.gfls}
\gfls u(x):=\textrm{p.v.}\int_{\R^n}g(|D_su|)\frac{D_su}{|D_su|}\frac{dy}{|x-y|^{n+s}}
\end{equation}
with $g=G'$ the derivative of a Young function  $G$, the quantity $D_s u(x,y):=\frac{u(x)-u(y)}{|x-y|^s}$ denotes the $s-$H\"older quotient and $d\mu:=\frac{dxdy}{|x-y|^n}$.

Eigenvalue problems for homogeneous operators of second order, the archetypal example being the $p$-Laplacian, have been widely studied and are by now fairly well understood, see \cite{AT,L}, or \cite{FBS2,Ka} for more general homogeneous operators. Non-homogeneous problems in Orlicz spaces, i.e. the (local) second order analog of \eqref{eq.eigen} have also been studied thoroughly, see for instance Gossez and Man\'asevich \cite{GM}, Tienari \cite{T}, Garc\'ia-Huidobro et al \cite{GLMS} and Mustonen and Tienari \cite{MT}, these last two references being a primary source of motivation for our present work.

On the other hand, the recent years have seen a tremendous development of the theory of nonlocal (or integro-differential) operators; such operators arise naturally in the context of stochastic L\'evy processes with jumps and have been studied thoroughly both from the point of view of Probability and Analysis as they proved to be accurate models to describe different phenomena in Physics, Finance, Image processing, or Ecology; see for instance \cite{Ap04,CT16,ST94} and references therein. For the mathematical background from the PDE perspective taken in this paper, see for instance \cite{BV} or \cite{Ga}.

The most canonical and important example of nonlocal operator is given by the \emph{fractional Laplacian}:
\[
(-\Delta)^s u(x) = \textrm{p.v.}\int_{\R^n} \frac{u(x)-u(y)}{|x-y|^{n+2s}}\, dy
\]
which can be obtained by minimizing the Gagliardo seminorm of the fracional Sobolev space $H^s$, $s\in(0,1)$ (see for instance \cite{DNPV}) or alternatively from the harmonic extension problem to the upper half space as the ``Dirichlet-to-Neumann'' operator, as pointed out in the celebrated paper of Caffarelli and Silvestre \cite{CS}. From this point of view, it is important to point out that when $G(t)=\frac{t^2}{2}$ and hence $g(t)=t$, then $\gfls$ becomes (a multiple of) $(-\Delta)^s$. More generally, when $G(t)=\frac{t^p}{p}$ in \eqref{eq.eigen}, $1<p<\infty$, we get the eigenvalue problem for the so called \emph{fractional $p$-Laplacian}
\[
(-\Delta_p)^s u(x) = \textrm{p.v.}\int_{\R^n} \frac{|u(x)-u(y)|^{p-2}(u(x)-u(y))}{|x-y|^{n+ps}}dy.
\]
Both of these problems have been studied recently, see for instance  \cite{BP, CL,DPS,FP, Kw, LL,SV} or the book \cite{BRS} and references therein.

A theory for eigenvalue problems in the context of fractional Orlicz-Sobolev spaces is still under development. The first author, together with  Fern\'andez Bonder, worked out the necessary functional setting for these type of problems in \cite{FBS}, and the eigenvalue problem was addressed very recently by the first author under the assumption of the $\Delta_2$, or doubling, condition (see Equation \eqref{eq.delta2} below) holds, see \cite{S}. However, a satisfactory result in the general setting was still lacking and is the main concern of this manuscript.   

It is worth mentioning that the lack of the doubling property, in general, carries a loss of reflexivity of the corresponding fractional Orlicz-Sobolev spaces. Its absence adds to the complications arising from the non-homogeneous nature of problem \eqref{eq.eigen}, making its solution a challenging task. 

\medskip

One way to look for (nontrivial) solutions to \eqref{eq.eigen} is to consider the constrained minimization problem
\begin{equation}\label{eq.min}
\min\left\{\F(u):u\in W^s_0L^G(\Omega),\,\J(u)=\alpha\right\},
\end{equation}
where   $W^s_0 L^G(\Omega)$ is a suitable fractional Orlicz-Sobolev space (see section \ref{sec.sobolev}) and
\begin{equation}\label{eq.modulars}
\F(u):=\iint_{\R^n\times\R^n} G(|D_su(x,y)|)\,d\mu,\quad\J(u):=\int_{\Omega}G(|u|)\,dx.
\end{equation}

Indeed, if we can find a minimizer $u_\alpha\in W^s_0L^G(\Omega)$ of \eqref{eq.min} then we may use the Lagrange multipliers method to ensure the existence of an eigenvalue $\lambda_\alpha$ such that \eqref{eq.eigen} is fulfilled by $u_\alpha$ in an appropriated weak sense. If $G$ satisfies the $\Delta_2$ condition, the Lagrange multipliers method applies directly as $\F$ and $\J$ are Fr\'echet differentiable in this case, see \cite{S}. See also \cite{BS} for a the Neumann case.  In this paper, by adapting the strategy of \cite{MT}, we get the existence of solutions of \eqref{eq.eigen} \emph{without assuming  the $\Delta_2$ condition neither on $G$ nor on its conjugated function.}

\subsection{Main results}

\medskip
We first prove that the constrained minimization problem \eqref{eq.min} has a solution for each energy level $\alpha>0$. Note that, since the functionals $\F$ and $\J$ are in general not homogeneous, minimizers strongly depend on the energy level $\alpha$. In all the following results $\Omega\subset \R^n$ is an open and bounded domain with Lipschitz boundary, $s\in(0,1)$ is a fractional parameter and $G$ stands for a Young function fulfilling the following structural conditions (see Section \ref{ssec.comp} for details)
$$
\int_1^\infty \left( \frac{t}{G(t)} \right)^\frac{s}{n-s}\,dt = \infty, \qquad \int_0^1 \left( \frac{t}{G(t)} \right)^\frac{s}{n-s}\,dt <\infty.
$$

\begin{thm}\label{thm.min}
For each $\alpha> 0$  there exists $u_\alpha\in  W^s_0L_G(\Omega)$, one-signed in $\Omega$, such that 
\[
\Lambda_\alpha:=\F(u_\alpha)=\min\left\{\F(u):u\in W^s_0L^G(\Omega),\:\J(u)=\alpha\right\}.
\]
\end{thm}

Once the existence of minimizers is achieved, we are lead to the following existence result for eigenvalues.
\begin{thm}\label{thm.wsol}
Let $u_\alpha\in W^s_0L^G(\Omega)$ be a solution of \eqref{eq.min}. Then there exists $\lambda_\alpha>0$ for which $u_\alpha$ is a weak solution of \eqref{eq.eigen}, i.e. 
\[
\iint_{\R^n\times\R^n} g\left(|D_su_\alpha|\right)\frac{D_su_\alpha}{|D_su_\alpha|}D_sv\,d\mu=\lambda_\alpha \int_{\Omega}g\left(|u_\alpha|\right)\frac{u_\alpha}{|u_\alpha|}v\,dx 
\]
for any $v\in  W^s_0L^G(\Omega)$.
\end{thm}

The next result states that the spectrum of the fractional $g-$Laplacian is closed in the following sense:

\begin{thm} \label{lam.inf}
Fix $\alpha_0>0$. With the notation of Theorem \ref{thm.wsol} we have that
$$
\Sigma:=\{ \lam:\eqref{eq.eigen}\text{ has a nontrivial solution } u \text{ satisfying } \Phi_G(u)\leq \alpha_0\}
$$
is a closed subset of $\R$.
\end{thm}

Theorem \ref{lam.inf} has an immediate corollary which we would like to point out:
\begin{cor}\label{cor}
Fix $\alpha_0>0$. With the notation of Theorem \ref{thm.wsol} let
$$
\hat{\lam}:=\inf\{ \lam_\alpha:0<\alpha\leq\alpha_0\}. 
$$

Then the eigenvalue problem for $\hat{\lambda}$ has a nontrivial solution. 
\end{cor}

In general we have that $\lam_\alpha>0$ and $\hat{\lam}> 0$. The following result refines the lower bound in thin domains.
\begin{prop} \label{prop.cota}
When $\Omega$ has small diameter $\d$ in the sense that for a fixed $\ve>0$
$$
\d\leq \left(\frac{1}{2(1+\ve)}\min\left\{1,\frac{n\omega_n}{2s}\right\} \right)^{\frac{1}{s}},
$$
then $\lam_\alpha \geq \ve$ and $\lam_0 \geq  \ve$, where $\omega_n$ denotes the volume of the unit sphere in $\R^n$.
\end{prop}

\begin{rem}
When $G$ and its Legendre transform $\tilde G$ satisfy the doubling condition, there exists a constant $c$ depending only on the growth behavior of $G$ such that
\begin{equation} \label{desig}
c^{-1} \Lambda_\alpha \leq \lam_\alpha \leq c \Lambda_\alpha,
\end{equation}
(see \cite[Corollary 5.3]{S}). However, for a general $G$ both quantities are not easily comparable. To find a relation like \eqref{desig} for a general $G$ with no $\Delta_2$ assumptions is left then as an open issue.
\end{rem}

Finally, the following Faber-Krahn type result follows in our setting.
\begin{prop}\label{prop.fk}
Let $B$ be a ball with $\mathcal{L}^n(B)=\mathcal{L}^n(\Omega)$ and $\alpha>0$. Then
$$
\Lambda_\alpha(B) \leq \Lambda_\alpha(\Omega).
$$
If additionally the application $t\mapsto tg(t)$ is convex, it holds that
$$
\lam_\alpha(B) \leq \lam_\alpha(\Omega).
$$
\end{prop}

\begin{rem}
The results in this paper could apply more generally to 
\begin{equation*}
\left\{ \begin{array}{cccc}
\gfls u  & = & \lambda h\left(|u|\right)\frac{u}{|u|} &  \textrm{ in }\Omega \\
u & = & 0  &\textrm{ in }\R^n\setminus\Omega,
\end{array} \right.
\end{equation*}
with $h$ the derivative of a Young function $H$ which grows essentially more slowly than $G_\ast$ (see \ref{ssec.comp} for definitions). However, we decided to stick with the simpler case for the sake of clarity in the presentation. 
\end{rem}

As an application, we would like to point out that slight modifications to the arguments in Sections \ref{sec.min} and \ref{sec.lag} allow us to prove existence of solutions for the following nonlinear eigenvalue problem:
\begin{equation}\label{eq.nonlinear}
\left\{ \begin{array}{cccc}
\gfls u &  = &\lambda f(u)  & \textrm{ in }\Omega \\
u & = & 0  & \textrm{ in }\R^n\setminus\Omega,
\end{array} \right.
\end{equation}
To set the problem properly, let us define the constraint functional
\[
\tJ(u):=\int_\Omega F(u)\:dx
\] 
with $F\colon \R \to \R$ a given locally Lipschitz function. We also set $f=F'$ and we will assume  
\begin{equation}\label{eq.f}
|f(t)|\leq C(|t|+1),
\end{equation}
and consider the minimization problem
\begin{equation}\label{eq.min2}
\min\left\{\F(u):u\in W^s_0L^G(\Omega),\,\tJ(u)=0\right\}.
\end{equation}

Then, we have the following Theorem:

\begin{thm}\label{thm.nonlinear}
Let $G$ be a Young function and let $F\colon\R \to\R$ be a locally Lipschitz function such that $f=F'$ satisfies \eqref{eq.f}. Then the minimization problem \eqref{eq.min2} has a solution $u_0\in W^s_0L^G(\Omega)$ and there exists $\lambda\in\R$ for which $u_0$ is a weak solution of the eigenvalue problem \eqref{eq.nonlinear}.  
\end{thm}

The rest of the paper is organized as follows: Section \ref{sec.orl} contains the necessary preliminary definitions and results regarding Young functions and Orlicz and fractional Orlicz-Sobolev spaces. Section \ref{sec.min} is devoted to the proof Theorem \ref{thm.min}, i.e. the existence of minimizers, and in Section \ref{sec.lag} we prove that minimizers are in fact solutions to the eigenvalue problem for some $\lambda>0$ (Theorem \ref{thm.wsol}). In Section \ref{sec.spec} we give the proof of Theorem \ref{lam.inf} and the further properties of the spectrum contained in Corollary \ref{cor} and Propositions \ref{prop.cota} and \ref{prop.fk}. Finally, in Section \ref{sec.nonlinear} we prove Theorem \ref{thm.nonlinear}.
 
\section{Orlicz and fractional Orlicz-Sobolev spaces}\label{sec.orl}

In this section we present the relevant aspects of the theory of Orlicz and fractional Orlicz-Sobolev spaces: the interested reader is referred to the classical reference \cite{Ad} and the recent papers \cite{Cianchi, DNFBS, FBS} for further treatment of these spaces. We start with the definition of \emph{Young functions.} 

\subsection{Young functions}\label{ssec.yfun}
An application $G\colon\R_+\to \R_+$ is said to be a  \emph{Young function} if it admits the integral formulation $G(t)=\int_0^t g(\tau)\,d\tau$, where the right continuous function $g$ defined on $[0,\infty)$ has the following properties:
\begin{align*}
&g(0)=0, \quad g(t)>0 \text{ for } t>0 \label{g0} \tag{$g_1$}, \\
&g \text{ is nondecreasing on } (0,\infty) \label{g2} \tag{$g_2$}, \\
&\lim_{t\to\infty}g(t)=\infty  \label{g3} \tag{$g_3$} .
\end{align*}
From these properties it is easy to see that a Young function $G$ is continuous, nonnegative, strictly increasing and convex on $[0,\infty)$. Also, from the convexity of $G$ it easily follows that
\begin{equation}\label{G1}
G(\alpha t) \leq \alpha G(t) \quad \text{ if } \alpha\in[0,1], \, t\geq 0
\end{equation}
and
\begin{equation}\label{G2} 
G(\beta t)\geq \beta G(t) \quad \text{ if } \beta\in(1,\infty), \, t\geq 0.
\end{equation} 
 
A Young function $G$ is said to satisfy the \emph{$\Delta_2$ condition (or doubling condition)} if there exist $C>0$ and $T\geq0$ such that 
\begin{equation}\label{eq.delta2}
G(2t)\leq CG(t)\quad \text{for all }t\geq T.
\end{equation}
 
In \cite{S}, the eigenvalue problem \eqref{eq.eigen} is addressed under the assumption that there exist $p^-,p^+>0$ such that 	
\begin{equation}\label{eq.p}
1<p^-\leq\frac{tg(t)}{G(t)}\leq p^+<\infty\quad\textrm{ for any }t>0.
\end{equation}
Roughly speaking, \eqref{eq.p} says that $G$ is trapped between to power functions.  The upper bound in this condition can be shown to be equivalent to \eqref{eq.delta2}, see Theorem 4.4.4 of \cite{LKJS}. The lower bound, on the other hand, corresponds to the complementary function $\tilde G$ (see \eqref{eq.comp} below) satisfying the doubling condition. 

\emph{In this paper we do not make any such assumption on $G$.} As we will point out throughout the article, this will pose several difficulties as many good properties of Orlicz and fractional Orlicz-Sobolev spaces strongly rely on the estimate \eqref{eq.delta2}. 

\medskip

Possible Young functions that fall into the scope of our paper (and do not satisfy \eqref{eq.p}) are
\begin{itemize}

\item[$\ast$] $G_1(t):=e^{t^\gamma}-1$, $\gamma> 1$;

\item[$\ast$] $G_2(t):=e^t-t-1$;

\item [$\ast$] $G_3(t)=\tau(e^t-t-1)+(1-\tau)\frac{t^p}{p}$ where $p\in(1,\infty)$ and $\tau\in(0,1)$ or, in general, any other convex combination of Young functions such as one of them does not satisfy \eqref{eq.p}. 
\end{itemize}

If we consider for instance $G_1$, Theorem \ref{thm.wsol} gives nontrivial (weak) solutions of the following eigenvalue problem
\[
\textrm{p.v.}\int_{\R^n}(u(x)-u(y))|u(x)-u(y)|^{\gamma-2}\frac{e^{\frac{|u(x)-u(y)|^\gamma}{|x-y|^{s\gamma}}}}{|x-y|^{s(\gamma-1)+n+1}}\:dy = \lambda  e^{|u|^\gamma}|u|^{\gamma-2}u. 
\]

The \emph{complementary Young function} $\tilde G$ of a Young function $G$ is defined as
\begin{equation}\label{eq.comp}
\tilde G(t):=\sup\{tw -G(w): w>0\}.
\end{equation}
It is not hard to see that $\tilde G$ can be written in terms of the inverse of $g$ as
\begin{equation} \label{eq.tildeg}
\tilde G(t)=\int_0^t g^{-1}(\tau)\,d\tau,
\end{equation}
see \cite{Ad}.

From \eqref{eq.comp} it is clear that the following Young-type inequality holds
\[
ab\leq G(a)+\tilde G(b)\qquad \text{for all }a,b\geq 0,
\]
and the following H\"older's type inequality
\begin{equation}\label{eq.hol}
\int_\Omega |uv|\,dx \leq 2\|u\|_G \|v\|_{\tilde G}
\end{equation}
for all $u\in L^G(\Omega)$ and $v\in L^{\tilde G}(\Omega)$. 

\medskip

The following relation are sometimes useful: for any $t>0$ 
\begin{equation} \label{rel}
G(2t)=\int_0^{2t} g(\tau)\,d\tau  > \int_t^{2t} g(\tau)\,d\tau > tg(t)
\end{equation}
and
\begin{equation} \label{rel2}
G(t)=\int_0^{t} g(\tau)\,d\tau  \leq t g(t).
\end{equation}

\subsection{Orlicz spaces} \label{sec.sobolev}

Given a set $\Omega\subset \R^n$ not necessarily bounded, we consider the Orlicz class $K^G(\Omega)$ defined as 
$$
K^G(\Omega)=\left\{ u \text { measurable and defined in }\Omega\colon \int_\Omega G(|u(x)|)\,dx<\infty\right\}.
$$
It is well-known that $K^G(\Omega)$ is a vector space if and only if $G$ satisfies the $\Delta_2$ condition.

The Orlicz space $L^G(\Omega)$ is the lineal hull of $K^G(\Omega)$, that is, the smallest vector space (under pointwise addition and scalar multiplication) that contains $K^G(\Omega)$. It follows that $L^G(\Omega)$ contains all scalar multipliers $\lam u$ of $u\in K^G(\Omega)$. Thus 
$$K^G(\Omega)\subset L^G(\Omega)$$ 
with equality if and only if $G$ satisfies the $\Delta_2$ condition.

The space $L^G(\Omega)$ is a Banach space endowed with the Luxemburg norm
\begin{equation}\label{eq.lnorm}
\|u\|_G=\inf\left\{ \lam>0\colon \int_\Omega G\left(\frac{|u(x)|}{\lam}\right)\,dx \leq 1\right\}.
\end{equation}
The closure in $L^G(\Omega)$ of all bounded measurable functions is denoted by $E^G(\Omega)$. It follows then that
$$E^G(\Omega)\subset K^G(\Omega)$$ 
with equality if and only if $G$ satisfies the $\Delta_2$ condition.

An important remark  is that the space $L^G$ is reflexive if and only if both $G$ and $\tilde G$ satisfy the $\Delta_2$ condition (see \cite{KR}).

\subsection{Fractional Orlicz-Sobolev spaces}\label{ssec.forl} 

Given a fractional parameter $s\in (0,1)$ we consider the space 
$$
W^s L^G(\Omega)=\left\{  u\in L^G(\Omega)\colon D_s u \in L^G (\R^n\times\R^n, d\mu)\right\}
$$
where we have denoted the $s-$H\"older quotient
$$
D_s u(x,y) = \frac{u(x)-u(y)}{|x-y|^s}
$$
and the measure $$d\mu=\frac{dxdy}{|x-y|^n}.$$  

From now on, the modulars in $L^G(\Omega)$ and $W^s L^G(\Omega)$ will be denoted as
$$
\Phi_G(u):=\int_\Omega G(|u|)\,dx \qquad \Phi_{s,G}(u):=\iint_{\R^n\times \R^n} G(|D_s u|)\,d\mu,
$$
respectively. Over the  space $W^s L^G(\Omega)$ we define the norm 
\begin{equation}\label{eq.norm}
\|u\|_{s,G} := \|u\|_G + [u]_{s,G},
\end{equation}
where
\[
[u]_{s,G} :=\inf\left\{\lambda>0\colon \Phi_{s,G}\left(\frac{u}{\lambda}\right)\le 1\right\}
\]
is the {\em $(s,G)-$Gagliardo seminorm}. The space $W^s E^G(\R^n)$ is defined in an analogous way.

The following structural properties hold true:
\begin{prop}
Let $\Omega$ be an open bounded subset of $\R^n$ with Lispchitz boundary and $G$ an Young function. Then
\begin{enumerate}
\item[(i)] $L^G(\Omega)$ is a Banach space under the norm \eqref{eq.lnorm} and $E^G(\Omega)$ is a closed subspace of $L^G(\Omega)$ (hence a Banach space itself). Furthermore, $E^G(\Omega)$ is separable. 

\item[(ii)] $W^sL^G(\Omega)$ is a Banach space with the norm \eqref{eq.norm}, $W^sE^G(\Omega)$ is a closed subspace of $W^sL^G(\Omega)$ (hence a Banach space itself). Furthermore, $W^sE^G(\Omega)$ is separable. 
\end{enumerate}
\end{prop}
The separability of $L^G(\Omega)$ is contingent on $G$ satisfying the $\Delta_2$ condition.

The space $W^s_0 E^G(\Omega)$ is the closure of $C_0^\infty(\Omega)$ in $W^sL^G(\Omega)$ with respect to the norm \eqref{eq.norm}. In view of \cite[Proposition 2.11]{FBS}, throughout the map 
$$
u\mapsto \left(u,D_s u \right)
$$
the spaces $W^s L^G(\Omega)$ and $W^s E^G(\Omega)$ can be isometrically identified with $L^G(\Omega)\times L^G(\R^n\times \R^n,d\mu)$ and $E^G(\Omega)\times E^G(\R^n\times \R^n,d\mu)$, respectively.
Hence, using the fact that
$$
(E^{\tilde G})'=L^G,\qquad \text{and}\qquad (E^G)' = L^{\tilde G},
$$
(see for instance \cite{Ad}), the space $W^s L^G(\Omega)$ is a closed subspace of $L^G(\Omega)\times L^G(\R^n\times \R^n,d\mu)$ being this space space the dual of the separable space $(E^{\tilde G}\times E^{\tilde G}(d\mu))'$. Therefore, by the Banach-Alaoglu theorem, $W^s L^G(\Omega)$ is weak* closed in $L^G(\Omega)\times L^G(d\mu)$.

We define $W^s_0 L^G(\Omega)$ as the weak* closure of $C_c^\infty(\Omega)$ in $W^s L^G(\Omega)$, hence $W^s_0 L^G(\Omega)$ is a weak* closed subset of the dual of a separable space.

\subsection{Compact embeddings}\label{ssec.comp}

In this subsection we recall some compact space embeddings. As mean of a comparative explanation, we recall that in the classical fractional Sobolev space setting $W^{s,p},\:s\in(0,1),\:p\in[1,\infty)$ the compact embedding
\begin{equation}\label{eq.sobemb}
 W^{s,p}(\Omega)\subset\subset L^q(\Omega)
\end{equation}
holds for $q\in[1,p]$ in the sense that any family $\mathcal{F}$ satisfying that 
\[
\sup_{u\in\mathcal{F}}\iint_{\Omega\times\Omega}\frac{|u(x)-u(y)|^p}{|x-y|^{n+sp}}\:dy\:dx<\infty
\]
is precompact in $L^q(\Omega)$ (see e.g. \cite[Theorem 7.1]{DNPV}). In particular, if $sp<n$, the continuous embedding allows to extend \eqref{eq.sobemb} to 
$q\in[1,p^\ast)$, where $p^\ast$ is the Sobolev conjugate 
\[
p^\ast:=\frac{np}{n-sp}.
\]

When trying to extend these notions to Orlicz-Sobolev spaces (and particularly when aiming at obtaining an analogous embedding as \eqref{eq.sobemb}), a precise notion of the behavior of the functions at infinity is required, notion that is self-evident in the classical case as powers are readily comparable. The following definitions serve that precise purpose: given Young functions $A$ and $B$, we say that $A$ decreases essentially more rapidly (near infinity) than $B$ if
$$
\lim_{t\to\infty} \frac{A(t)}{B(\beta t)}=0
$$
for all $\beta>0$. In particular, there exists positive constants $c$ and $T$ such that 
$$
A(t)\leq B(ct) \quad\text{for } t\geq T.
$$

Let $G$ be a Young function such that
\begin{equation} \label{condi}
\int_1^\infty \left( \frac{t}{G(t)} \right)^\frac{s}{n-s}\,dt = \infty, \qquad \int_0^1 \left( \frac{t}{G(t)} \right)^\frac{s}{n-s}\,dt <\infty,
\end{equation}
we define the critical function $G_*(t):=G(H^{-1}(t))$, where
$$
H(t)=\left( \int_0^t \left(\frac{\tau}{G(\tau)}\right)^\frac{s}{n-s} \,d\tau \right)^\frac{n-s}{n}.
$$

The appropriate compact embedding is contained in the following Theorem, whose proof can be found in \cite[Theorem 6.1]{Cianchi}.
\begin{thm} \label{embedding}
Let $G$ be a Young function satisfying \eqref{condi}. If the Young function $B$ grows essentially more slowly than $G_*$ near infinity, then the embedding $$W^s L^G(\Omega)\subset L^B(\Omega)$$ is compact for every bounded Lipschitz domain $\Omega$ in $\R^n$.
\end{thm}

In particular we can take $B=G$ in the previous theorem giving the following.
\begin{prop} \label{embed}
Under the hypotheses of the previous theorem,  the embedding of $W^sL^G(\Omega)$ into $L^G(\Omega)$ is compact.
\end{prop}

\begin{proof}
Since $H(t)$ is increasing, given a fixed $\beta>0$ we have
$$
\lim_{t\to\infty} \frac{G( t)}{G(H^{-1}(\beta t))} = \lim_{t\to\infty} \frac{G(\beta^{-1} H(t))}{G(t)} .
$$
Since $t/G(t)$ is non-increasing,
$$
H(t)\leq \left(\int_0^t \frac{d\tau}{G(d\tau)^\frac{s}{n-s}} \right)^\frac{n-s}{s}\leq G(1)^{-\frac{s}{n}}  t^{1-\frac{s}{n}} 
$$
from where
\begin{align*}
\lim_{t\to\infty} \frac{G( t)}{G(H^{-1}(\beta t))} &\leq 
\lim_{t\to\infty} \frac{G(c t^{1-\frac{s}{n}})}{G(t)} = c(1-\frac{s}{n}) \lim_{t\to\infty} t^{-\frac{s}{n}} \frac{g(c t^{1-\frac{s}{n}})}{g(t)} \\
&\leq  c(1-\frac{s}{n})\lim_{t\to\infty} t^{-\frac{s}{n}} =0
\end{align*}
where $c=\beta^{-1} G(1)^{-\frac{s}{n}}$.

Therefore the result follows from Theorem \ref{embedding}.

\end{proof}

\subsection{A Poincar\'e's inequality}

The following Poincar\'e inequality will be of use, and has independent interest:
\begin{prop} \label{eq.poincare}
Let $\Omega\subset \R^n$ be open and bounded and let $G$ be a Young function. Then for $s\in(0,1)$ it holds that
$$
\Phi_G(u) \le   \Phi_{s,G}(C \d^s u)
$$
and
$$
\|u\|_G \leq C \d^s [u]_{s,p}
$$
for all $u\in W^s_0 L^G(\Omega)$, where $C=C(n,s)$ and $\d$ denotes the diameter of $\Omega$.
\end{prop}

\begin{proof}
Given $x\in \Omega$, observe that when $|x-y|\geq \d$, then $y\notin \Omega$. Hence 
$$
\Phi_{s,G}(u) \geq 
\int_\Omega \int_{\d^s \leq |x-y|\leq 2\d^s} G\Big( \frac{  \d^s  }{|x-y|^s} \frac{ |u(x)|}{\d^s}  \Big) \frac{dydx}{|x-y|^n}.
$$
Since $\frac{2\d^s}{|x-y|^s}\geq 1$, from \eqref{G2} we get that the expression above is greater than
$$
\int_\Omega \int_{\d^s \leq |x-y|\leq 2\d^s} G\Big( \frac{|u(x)|}{2\d^s}  \Big)  \d^s \frac{dydx}{|x-y|^{n+s}},
$$
which can be written as
$$
\int_\Omega  G\Big( \frac{|u(x)|}{2\d^s}  \Big) \,dx    \int_{\d^s \leq |z|\leq 2\d^s} \d^s \frac{dz}{|z|^{n+s}},
$$
and moreover, using polar coordinates it is equal to 
$$
 \frac{n\omega_n}{2s} \int_\Omega  G\Big( \frac{|u(x)|}{2\d^s}  \Big) \,dx.
$$
The last four expressions lead to
$$
\Phi_{s,G}(u) \geq  \frac{n\omega_n}{2s} \int_\Omega  G\Big( \frac{|u(x)|}{2\d^s}  \Big) \,dx.
$$
Finally, by using \eqref{G2}, if we denote $C^{-1}:= \min\{1,\frac{n\omega_n}{2s}\}$ we get
$$
\Phi_{s,G}(u) \geq   \int_\Omega  G\Big( \frac{|u(x)|}{2C\d^s}  \Big) \,dx,
$$
or, equivalently, $\Phi_G(u) \leq \Phi_{s,G}\left(2C\d^s u\right)$. Finally, since
$$
\Phi_G\left(\frac{u}{2C\d^s[u]_{s,G}}\right) \leq \Phi_{s,G}\left(\frac{u}{[u]_{s,G}}\right),
$$
by definition of the Luxemburg norm we obtain that $\|u\|_G \leq 2C\d^s [u]_{s,G}$ and the proof concludes.
\end{proof}
 
A standard consequence of the Poincar\'e inequality is the following:
\begin{cor}
In light of Proposition \ref{eq.poincare}, $[\cdot]_{s,G}$ is an equivalent norm in $W^s_0 L^G(\Omega)$.
\end{cor} 

\section{The minimization problem} \label{sec.min}

In this section we prove Theorem \ref{thm.min}; notice that the functionals $\F$ and  $\J$ coincide with the modulars defining the Orlicz spaces:
$$
\F(u)=\Phi_{s,G}(u), \qquad \J(u)=\Phi_G(u).
$$

We rewrite slightly \eqref{eq.min} as follows: let
\[
M_\alpha:=\{ u \in W^s_0 L^G (\Omega)\colon \J(u)=\alpha\}
\]
and 
\begin{equation} \label{minim.prob}
\Lambda_\alpha=\min_{u\in M_\alpha} \F(u).
\end{equation}

We start proving some properties on the functional $\F$.
\begin{prop} \label{prop.coerciva}
The functional $\F$ is coercive.
\end{prop}

\begin{proof}
Assume that $[u]_{s,G}>1+\ve$ for some  $\ve>0$, then from $(G_1)$ and the definition of the Luxemburg norm we get
$$
\frac{1+\ve}{[u]_{s,G}} \Phi_{s,G}(u) \geq \Phi_{s,G} \left( \frac{(1+\ve)u}{[u_{s,G}]} \right) >1,
$$
from where
$$
\Phi_{s,G}(u) > \frac{[u]_{s,G}}{1+\ve}.
$$
Hence, by the arbitrariness of $\ve$ we get that $\Phi_{s,G}(u)\geq [u]_{s,G}$ for all $u\in W^s_0 L^G(\Omega)$, giving the desired coercivity of $\F$.
\end{proof}

\begin{prop} \label{prop.wlsc}
The functional $\F$ is weak* lower semicontinuous.
\end{prop}
\begin{proof}
Consider a sequence $\{u_k\}_{k\in\N}\subset W^s_0 L^G(\Omega)$ such that $u_k \cde u$ in $W^s_0 L^G(\Omega)$. Since $W^s_0 L^G(\Omega)$ is a weak* closed subspace of the dual of a separable space, we have the following convergences
$$
\int_\Omega u_k \varphi\,dx \to \int_\Omega u \varphi\,dx, \qquad \langle (-\Delta_g)^s u_k,\varphi\rangle \to \langle (-\Delta_g)^s u,\varphi\rangle \qquad \forall \varphi \in E^{\tilde G}(\Omega).
$$
In particular, this holds for all $\varphi \in L^\infty(\Omega)$ and hence $u_k \cd u$ weakly in $W^s L^1(\Omega)$.

Since $\{u_k\}_{k\in\N}$ is weakly* convergent, it is   bounded  in $W^s_0 L^G(\Omega)$, therefore due to the compact embedding given in Proposition \ref{embed}, there exists $v\in W^s_0 L^G(\Omega)$ such that
\begin{align*}
u_k \longrightarrow v &\quad \text{ strongly  in } L^G(\Omega),\\
u_k \longrightarrow v &\quad \text{ a.e. in } \R^n.
\end{align*}
As a consequence, we conclude that $u=v$ and obtain that
$$
G(|D_s u_k|) \longrightarrow G(|D_s u|) \quad \text{ a.e. in } \R^n.
$$
Then, by using Fatou's Lemma we conclude that
$$
\F(u)=\iint_{\R^n\times\R^n} G(|D_s u|)\, d\mu \leq 
\liminf_{k\to\infty }\iint_{\R^n\times\R^n} G(|D_s u_k|)\, d\mu =\liminf_{k\to\infty} \F(u_k)
$$
and the proof finishes.
\end{proof}

The following two lemmas are the key to prove that the set $M_\alpha$ is sequentially weak* closed. Observe that $K^G(\Omega)$ in general is not a vector space when $G$ does not satisfy the $\Delta_2$ condition, however, the following holds:

\begin{lem}\label{lema.aux.1}
Given a Young function $G$ and $v\in W^s L^G(\Omega)$, then $\delta v \in K^G(\Omega)$ for every $\delta>0$.
\end{lem}
\begin{proof}
Let $v\in W^s L^G(\Omega)$ and let $\delta>0$ be fixed. In view of Proposition \ref{embedding}, $v\in L^{G_*}(\Omega)$. Moreover, by Proposition \ref{embed}, $G$ grows essentially more  slowly than $G_*$, from where we can choose positive constants $K$ and $T$ such that $\|v\|_{G_*} \leq K$ and
$$
G(t) \leq G_*\left( \frac{t}{\delta K} \right) \quad \text{for } t\geq T.
$$
Consider now the subset 
$$
\Omega_K =\left\{  x\colon v(x)\leq \frac{T}{\delta}\right\}.
$$
Then, these observations together with the definition of the Luxemburg norm give
\begin{align*}
\int_\Omega G(\delta |v(x)|)\,dx &\leq \int_{\Omega_K} G(\delta|v(x)|)\,dx + \int_{\Omega\setminus \Omega_K} G_*\left( \frac{|v(x)|}{K} \right)\,dx\\
&\leq G(T) |\Omega| + \int_\Omega G_*\left( \frac{|v(x)|}{K} \right)\,dx\\
&\leq G(T) |\Omega| + 1,
\end{align*}
that is, $\delta v\in K^G(\Omega)$.  
\end{proof}

\begin{lem} \label{lema.aux}
Given a Young function $G$ and $v\in W^s L^G(\Omega)$, we have that 
$$
\int_\Omega \tilde G(g(|v|)\,dx < \infty, 
$$
that is, $g(|v|) \in K^{\tilde G}(\Omega)$.
\end{lem}

\begin{proof}
Let $v\in W^s L^G(\Omega)$. By Lemma \ref{lema.aux.1} it follows that $2v \in K^G(\Omega)$. Therefore, using \eqref{eq.tildeg}, the fact that $g$ is increasing and \eqref{rel} we get
\begin{align*}
\int_\Omega \tilde G(g(|v|))\,dx &\leq \int_\Omega \left( \int_0^{g(|v|)} g^{-1}(\tau)\,d\tau \right)\,dx\\
&\leq \int_\Omega |v|g(|v|)\,dx \leq \int_\Omega G(2|v|)\,dx
\end{align*}
giving that $g(|v|)\in K^{\tilde G}(\Omega)$.
\end{proof}

\begin{prop} \label{prop.mmu}
The set $M_\alpha$ is sequentially weak* closed.
\end{prop}
\begin{proof}
Given $\{u_k\}_{k\in\N}\subset M_\alpha$ such that $u_k \cde u$ let us see that $\J(u)=\mu$.

Observe that, since $g$ is increasing 
\begin{align*}
|\J(u_k)-\J(u)|&\leq \int_\Omega |G(|u_k|)-G(|u|)|\,dx\\
&\leq \int_\Omega \left(\int_{|u(x)|}^{|u(x)|+|u_k(x)-u(x)|} g(t)\,dt \right)dx\\
&\leq \int_\Omega g(|u(x)|+|u_k(x)-u(x)|)|u_k(x)-u(x)|\,dx.
\end{align*}
Gathering the inequality above with the H\"older's inequality for Young functions and    Lemma \ref{lema.aux}, we get
\begin{equation} \label{una.ecuacion}
|\J(u_k)-\J(u)| \leq  2 \| g(|u|+ |u_k-u|)\|_{\tilde G} \|u_k-u\|_G. 
\end{equation}

Define $\bar u_k(x):=|u_k(x)-u(x)|+|u(x)|$. We claim that $g(\bar u_k)$ is uniformly bounded in $L^{\tilde G}(\Omega)$. Indeed, if $\|g(\bar u_k)\|_{\tilde G}\leq 2$ there is nothing to do. Assume otherwise that $\|g(\bar u_k)\|_{\tilde G}>2$, then by the convexity of $\tilde G$ and the definition of the Luxemburg norm we have that
$$
1<\int_\Omega \tilde G\left(\frac{2g(\bar u_k)}{\|g(\bar u_k)\|_{\tilde G}} \right)\,dx \leq \frac{2}{\|g(\bar u_k)\|_{\tilde G}} \int_\Omega \tilde G(g(\bar u_k)) \,dx,
$$
from where, using \eqref{eq.tildeg}, the fact that $g$ is increasing and  \eqref{rel} we get
\begin{align*}
\frac{1}{2} \|g(\bar u_k)\|_{\tilde G} &\leq \int_\Omega \tilde G(g(\bar u_k))\,dx\\
&\leq \int_\Omega \left( \int_0^{\bar g(u_k)} g^{-1}(\tau)\,d\tau\right)\,dx\\
&\leq \int_\Omega \bar u_k g(u_k)\,dx \leq \int_\Omega G(2\bar u_k)\,dx.
\end{align*}
Observe that, for any $\ve>0$, 
$$
\int_\Omega G(2\bar u_k)\,dx \leq \frac{1}{1+\ve}\int_\Omega G\left( (1+\ve) |u_k-u| \right)\,dx + \frac{\ve}{1+\ve}\int_\Omega G\left(\frac{(1+\ve)}{\ve} |u| \right)\,dx,
$$
then, since $u$ and $u_k-u$ belong to $W^s L^G(\Omega)$, from Lemma \ref{lema.aux.1} it follows that  $(1+\ve)|u_k-u|$ and $\frac{(1+\ve)}{\ve}|u|$ belong to $K^G(\Omega)$. Therefore giving from the previous considerations that $g(\bar u_k)$ is uniformly bounded in $L^{\tilde G}(\Omega)$.

The last assertion together with \eqref{una.ecuacion} gives that $\J(u_k)\to \J(u)$. Hence $\J(u)=\alpha$, $u\in M_\alpha$ and $M_\alpha$ is weak* closed.
\end{proof}

We are in position to prove that the minimization problem \eqref{minim.prob} has a solution.

\begin{proof}[Proof of Theorem \ref{thm.min}]
Since $G$ is increasing it follows that $M_\alpha$ is not empty. Let $\{u_k\}_{k\in\N}\subset M_\alpha$ be a minimizing sequence for \eqref{minim.prob}, that is, 
$$
\lim_{k\to \infty}   \F(u_k) = \Lambda_\alpha.
$$
From Proposition \ref{prop.coerciva} $f$ is coercive, implying  that $\{u_k\}_{k\in\N}$ is bounded in $W^s_0 L^G(\Omega)$, which is the dual of a separable Banach space. Hence, up to a subsequence, there exists $u\in W^s_0 L^G(\Omega)$ such that $u_k \cde u$ in $W^s_0 L^G(\Omega)$. Moreover, $u\in M_\alpha$ in light of Proposition \ref{prop.mmu}. Finally, by the weak* lower semicontinuity of $\F$ stated in Proposition \ref{prop.wlsc} we conclude that
$$
\F(u) \leq \liminf_{k\to\infty} \F(u_k) = \Lambda_\alpha,
$$
that is, $u$ is a solution of the constrained minimization problem \eqref{minim.prob}.

Finally, observe that $u$ can be assumed to be one-signed in $\Omega$ since if $u$ solves \eqref{minim.prob}, hence also $|u|$ does it.
\end{proof}

\section{Lagrange multipliers and the eigenvalue problem}\label{sec.lag} 

In this section we essentially follow \cite{MT} to prove Theorem \ref{thm.wsol}. The idea is to use a Lagrange multipliers type theorem to get existence of eigenvalues. Since we are not assuming the $\Delta_2$ condition, $\F$ and $\J$ are not, in general, differentiable. To overcome this, we have the following result.

\begin{prop}\label{prop.ds}
Let $u_\alpha\in W^s_0L^G(\Omega)$ be a solution of the minimization problem \eqref{eq.min}. Then $g(|D_su_\alpha|)\in L^{\tilde{G}}(\R^n\times \R^n, d\mu)$.
\end{prop}

\begin{proof}

It follows exactly as in the proof of Proposition 4.3 in \cite{MT} replacing $\nabla$ by $D_s$ and the Lebesgue measure by $d\mu$. 
\end{proof}

The following Lemma is also proved in \cite{MT}:
\begin{lem}\label{lem.2}
Let $u,v\in E^G(\Omega)$ such that $u\not\equiv0$ and $\int_\Omega g(|u|)v\,dx\neq0$. Then
\[
\int_\Omega G((1-\varepsilon)u+\delta v)\,dx=\int_\Omega G(|u|)\:dx  
\]
defines $\delta$ as a function of $\varepsilon$ on some interval $(-\varepsilon_0,\varepsilon_0)$. Moreover, $\delta$ is differentiable in $(-\varepsilon_0,\varepsilon_0)$, $\delta(0)=0$ and 
\[
\delta'(0)=\frac{\int_\Omega g(|u|)u\,dx}{\int_\Omega g(|u|)v\,dx}.
\]
\end{lem}

With these preliminaries we are ready to prove our eigenvalue existence result.

\begin{proof}[Proof of Theorem \ref{thm.wsol}]
Let $u=u_\alpha$ be a solution of \eqref{eq.min} and let us define $\F'$ and $\J'$ as linear functionals by:
\begin{align*}
\F'(v)&:=\iint_{\R^n\times\R^n} g\left(|D_su|\right)\frac{D_su}{|D_su|}D_sv\,d\mu,\\
\J'(v)&:=\int_{\Omega}g\left(|u|\right)\frac{u}{|u|}v\,dx.
\end{align*}
Notice that Proposition \ref{prop.ds} assures that 
\[
|\F'(v)|\leq \iint_{\R^n\times \R^n} \tilde{G}\left(g(|D_su_\alpha|)\right)\,d\mu+\iint_{\R^n\times \R^n} G(|D_sv|)\,d\mu<\infty 
\]
so that $\F'$ is well defined in $W^s_0E^G(\Omega)$. Similarly, $\J'$ is well defined by Lemma \ref{lema.aux}. 

According to \cite{Zei}, Proposition 43.1, if we show that 
\begin{equation}\label{eq.kerfkerg}
\F'(v)=0\Rightarrow \J'(v)=0
\end{equation}
then there exists $\lambda=\lambda_\alpha$ such that 
\[
\F'(v)=\lambda\J'(v)
\]
and we have a weak solution of \eqref{eq.eigen} from the density of $W^s_0E^G(\Omega)$ in $W^s_0L^G(\Omega)$.\\ 
Observe that the number $\lambda$ is strictly positive, fact which  follows just by taking $u$ itself as a test function.

To show \eqref{eq.kerfkerg} it is enough to show that 
\[
\J'(v)>0\Rightarrow \F'(v)>0
\]
so let $\J'(v)>0$ for some $v\in W^s_0E^G(\Omega)$ and let us define the function
\begin{equation}\label{eq.delta}
\int_\Omega G((1-\varepsilon)u+\delta(\varepsilon)v)\,dx=\alpha.
\end{equation}
Now, according to Lemma \ref{lem.2}, $\delta\in C^1(-\varepsilon_0,\varepsilon_0)$, $\delta(0)=0$ and $\delta'(0)>0$. Hence, 
\[
\frac{1}{2}\delta'(0)\leq\delta'(\varepsilon)\leq 2\delta'(0) \text{ for } \varepsilon\in(0,\varepsilon_1)
\]  
for $\varepsilon_1<\varepsilon_0$ small enough, so integrating we get
\begin{equation}\label{eq.deltaprime}
\frac{1}{2}\delta'(0)\leq \frac{\delta(\varepsilon)}{\varepsilon}\leq 2\delta'(0)\text{ for } \quad(0,\varepsilon_1).
\end{equation}
 Let $w_\varepsilon:=(1-\varepsilon)u+\delta(\varepsilon)v$ and notice that, since $w_\varepsilon$ is admissible, 
\begin{equation}\label{eq.int}
\iint_{\R^n\times \R^n} \frac{G(|D_sw_\varepsilon|)-G(|D_su|)}{\delta(\varepsilon)}\,d\mu\geq 0\quad\text{for any}\quad\varepsilon\in(0,\varepsilon_1).
\end{equation}

We want to take the limit on \eqref{eq.int} as $\varepsilon\rightarrow0^+$. Now, since $w_\varepsilon\longrightarrow u$ a.e. as $\varepsilon\rightarrow0^+$ we have
\begin{equation}\label{eq.dsw}
D_sw_\varepsilon \longrightarrow D_su  \quad \mu-\text{a.e.}
\end{equation}
and therefore, since $G$ is a continuous function
\[
G(|D_sw_\varepsilon|)\longrightarrow G(|D_su|) \quad \mu-\text{a.e.}
\] 
Also, by using the mean value theorem
\begin{align*}
\frac{G(|D_sw_\varepsilon|)-G(|D_su|)}{\delta(\varepsilon)} & =\frac{1}{{\delta(\varepsilon)}}\frac{G(|D_sw_\varepsilon|)-G(|D_su|)}{|D_sw_\varepsilon|-|D_su|}\frac{|D_sw_\varepsilon|^2-|D_su|^2}{|D_s w_\ve| + |D_s u|}\\       & = g(\xi(\varepsilon))\left(\frac{1}{{\delta(\varepsilon)}}\frac{|D_sw_\varepsilon|^2-|D_su|^2}{|D_sw_\varepsilon|+|D_su|}\right)
\end{align*}
where
\[
g(\xi(\varepsilon))\longrightarrow g(|D_su|)\quad\text{as}\quad \varepsilon\rightarrow0^+.
\]
Next, notice that by the definition of $w_\varepsilon$ 
\[
\frac{1}{{\delta(\varepsilon)}}\left(|D_sw_\varepsilon|^2-|D_su|^2\right)=\frac{\varepsilon(\varepsilon-2)}{\delta(\varepsilon)}|D_su|^2+2(1-\varepsilon)D_suD_sv+\delta(\varepsilon)D_sv.
\]
This, \eqref{eq.deltaprime} and \eqref{eq.dsw} give 
\[
\frac{1}{{\delta(\varepsilon)}}\frac{|D_sw_\varepsilon|^2-|D_su|^2}{|D_sw_\varepsilon|+|D_su|} \longrightarrow\frac{-1}{\delta'(0)}|D_su|+\frac{D_suD_sv}{|D_su|}
\]
 hence
\begin{equation}\label{eq.conv}
\frac{G(|D_sw_\varepsilon|)-G(|D_su|)}{\delta(\varepsilon)} \longrightarrow g(|D_su|)\left(\frac{-1}{\delta'(0)}|D_su|+\frac{D_suD_sv}{|D_su|}\right)
\end{equation}
$\mu$-a.e. as $\varepsilon\rightarrow0^+$. Furthermore
\begin{align*}
\left|\frac{G(|D_sw_\varepsilon|)-G(|D_su|)}{\delta(\varepsilon)}\right| & \leq \left(g(|D_sw_\varepsilon|)+g(|D_su|)\right)\frac{|D_sw_\varepsilon-D_su|}{\delta(\varepsilon)} \\  
& \leq \left(2g(|D_su|)+g\left(\frac{\delta(\varepsilon)}{\varepsilon}|D_sv|\right)\right)\left(\frac{\varepsilon}{\delta(\varepsilon)}|D_su|+|D_sv|\right) \\  
& \leq \left(2g(|D_su|)+g\left(2\delta'(0)|D_sv|\right)\right)\left(\frac{2}{\delta'(0)}|D_su|+|D_sv|\right) 
\end{align*}
and this function (again using Proposition \ref{prop.ds}) belongs to $L^1(\R^n\times \R^n, d\mu)$.\\
Then, by the Dominated Convergence Theorem, \eqref{eq.conv} holds in $L^1(\R^n\times \R^n, d\mu)$, and recalling \eqref{eq.int} we get 
\[
\iint_{\R^n\times \R^n} g(D_su)\frac{D_suD_sv}{|D_su|}\,d\mu\geq \frac{1}{\delta'(0)}\iint_{\R^n\times \R^n} g(D_su)D_su\,d\mu>0,
\]
that is, $\F'(v)>0$, which concludes the proof.
\end{proof}  

\section{Some properties of the spectrum}\label{sec.spec}

In this section we proof the properties of the spectrum discussed in the Introduction. We start with Theorem \ref{lam.inf}.

\begin{proof}[Proof of Theorem \ref{lam.inf}]
Let  $\{\lam_k\}_{k\in\N}$ be sequence of eigenvalues of \eqref{eq.eigen} such that $\lam_k\to \lam$ and let $\{u_k\}_{k\in\N}\subset W^s_0 L^G(\Omega)$ be the corresponding sequence of associated eigenfunctions, i.e., 
\begin{equation} \label{eqj}
\iint_{\R^n\times\R^n} g(|D_s u_k|) \frac{D_s u_k}{|D_s u_k|}  D_s v\,d\mu = \lam_k \int_\Omega g(|u_k|)\frac{u_k}{|u_k|} v \qquad \forall v\in W^s_0 L^G(\Omega).
\end{equation}

Arguing as in the proof of  Proposition  \ref{prop.wlsc}, up to a subsequence,  there exists $u\in W^s_0L_G(\Omega)$ such that
\[
\begin{array}{ll}
u_k\longrightarrow u &\text{ strongly  in }L^{G}( \Omega),\\
u_k\longrightarrow u &\text{ a.e. in } \R^n.
\end{array}
\]
From the continuity of $t\mapsto g(t)\frac{t}{|t|}$ we deduce that
$$
g(|D_s u_k|)\frac{D_s u_k}{|D_s u_k|}\longrightarrow g(|D_s u|)\frac{D_s u}{|D_s u|}\quad  \text{ a.e. in } \Omega
$$
and hence, taking limit as $k\to\infty$ in \eqref{eqj} we obtain that
$$
\iint_{\R^n\times\R^n} g(|D_s u|)\frac{D_s u}{|D_s u|}   D_s v \, d\mu = \lam \int_\Omega g(|u|)\frac{u}{|u|} v \qquad \text{ for all } v\in W^s_0 L^G(\Omega)
$$
from where the proof concludes.
\end{proof}

And we can also give the proof of Corollary \ref{cor}.

\begin{proof}[Proof of Corollary \ref{cor}]
It follows immediately  from Theorems \ref{thm.wsol} and \ref{lam.inf} by taking a sequence $\{\lambda_{\alpha_k}\}_{k\in\N}$ such that
\[
\lambda_{\alpha_k}\longrightarrow\hat{\lambda}
\] 
as $k\to\infty$, concluding the proof.
\end{proof}

Next we proof our uniform bound from below for eigenvalues in thin domains:

\begin{proof}[Proof of Proposition \ref{prop.cota}]
Observe that for each $\ve>0$, since $g$ is increasing we get
$$
G((1+\ve)t)=\int_0^{(1+\ve)t} g(\tau)\,d\tau  > \int_t^{(1+\ve)t} g(\tau)\,d\tau > \ve tg(t).
$$
This, together with \eqref{rel2} gives that
$$
\ve \frac{\Phi_{s,G}(u_\alpha)}{\Phi_G((1+\ve) u_\alpha)} \leq \lam_\alpha \leq  \frac{1}{\ve}
\frac{\Phi_{s,G}((1+\ve)u_\alpha)}{\Phi_G(u_\alpha)}.
$$
In particular, by using Proposition \ref{eq.poincare} we get
$$
\lam_\alpha \geq \ve \frac{\Phi_{s,G}(u_\alpha)}{\Phi_{s,G}((1+\ve) C \d^s u_\alpha)}.
$$
Hence, when $\Omega$ is such that its diameter is small enough, let us say $\d\leq ((1+\ve)C)^{-\frac{1}{s}}$, we find that $\lam_\alpha \geq \ve$. 

Since this lower bound in independent on $\alpha$ we obtain that $\hat{\lam}\geq\ve$ on these kind of sets. 
\end{proof}

Finally, we prove Proposition \ref{prop.fk}.

\begin{proof}[Proof of Proposition \ref{prop.fk}]
The first inequality is just a consequence of the P\'olya-Szeg\"o principle stated in \cite[Theorem 3.1]{Cianchi}.

On the other hand, let $u\in M_\alpha$ and let $u^*$ be the symmetric rearrangement of $u$. Assuming that $H(t):=tg(t)$ is convex, by applying again \cite[Theorem 3.1]{Cianchi} we have that
$$
\langle (-\Delta_g)^s u^*, u^* \rangle = \Phi_{s,H} (u^*) \leq \Phi_{s,H} (u) = \langle (-\Delta_g)^s u, u \rangle,
$$
from where, since $\Phi_H(u)=\Phi_H(u^*)=\alpha$, $\lam_\alpha(B) \leq \lam_\alpha(\Omega)$.
\end{proof}

\section{Nonlinear eigenvalue problem}\label{sec.nonlinear}

In this section we show how some minor modifications of the arguments in Section \ref{sec.min} and \ref{sec.lag} lead to the existence of the nonlinear eigenvalue problem \eqref{eq.nonlinear}.

\begin{proof}[Proof of Theorem \ref{thm.nonlinear}]
Existence of a solution for the minimization problem follow exactly as in the proof of Theorem \ref{thm.min} as long as we can show the stability of the constraint (notice that the functional we are minimizing is the same). Now, thanks to Proposition \ref{embed} we have that (a subsquence of) the minimizing sequence converges to $u_0$ in $L^G(\Omega)$. Then
\begin{align*}
|\tJ(u_0)| & = |\tJ(u_0) - \tJ(u_k)| \\
		   & \leq C\int_\Omega (|u_0|+|u_k|+1)|u_0-u_k| \,dx\\
		   & \leq C\int_\Omega |u_0-u_k| \,dx \\
  		   & \leq C\|u_0-u_k\|_G \,dx.
\end{align*}
Here the constant $C$ is changing from line to line and we have used \eqref{eq.f}, the fact that boundedness in norm implies boundedness in mean (see \cite{KR}) and the H\"older inequality \eqref{eq.hol}. Since the last term of the previous inequalities goes $0$ as $k\rightarrow\infty$, we have that $\tJ(u_0)=0$ as desired. Hence, $u_0$ is a minimizer. 

Now we want to show that  $u_0$ is a weak solution of the nonlinear eigenvalue problem \eqref{eq.nonlinear} for some $\lambda\in\R$. 

Let us separate two cases: 
\begin{enumerate}

\item $f(u_0)=0$ a.e.

\item $f(u_0)\neq 0$ a.e.
\end{enumerate}
 
If $(1)$ holds, then it is not hard to see that $u_0$ must vanish identically, and it is therefore a solution to \eqref{eq.nonlinear} for any $\lambda\in\R$. 

If $(2)$, we want to proceed as in Theorem \ref{thm.wsol}; a careful look at its proof   shows that the important properties of $\delta$ is that \eqref{eq.delta} holds and that
\begin{equation}\label{eq.deltabis}
\delta\in C^1(-\varepsilon_0,\varepsilon_0),\:\delta(0)=0\text{ and }\delta'(0)>0.
\end{equation}

Now, $(2)$ allows us to find $v\in W^s_0L^G(\Omega)$ such that 
\[
\int_\Omega f(u_0)v\:dx> 0.
\]
Let then consider
\[
h(\varepsilon,\delta):=\int_\Omega F((1-\varepsilon) u_0+\delta v)\:dx
\]
so that 
\[
h(0,0)=0\quad\text{ and }  \quad\frac{\partial h}{\partial \delta}(0,0)\neq0.
\]
Then, by the Implicit Function Theorem we can write $\delta=\delta(\varepsilon)\in C^1(-\varepsilon_0,\varepsilon_0)$ and 
\[
h(\varepsilon,\delta(\varepsilon))=0\quad\text{ for }\varepsilon\in(-\varepsilon_0,\varepsilon_0)
\]
and 
\[
\delta'(0)=-\frac{\frac{\partial h}{\partial \varepsilon}(0,0)}{\frac{\partial h}{\partial \delta}(0,0)}=\frac{\int_\Omega f(u_0)u_0\,dx}{\int_\Omega f(u_0)v\,dx}>0
\]
and we get \eqref{eq.delta} and \eqref{eq.deltabis}. The rest of the proof follows as in the proof of Theorem \ref{thm.wsol}.
\end{proof}

\subsection*{Acknowledgements.} This work was partially supported by Consejo Nacional de Investigaciones Cient\'{i}ficas y T\'{e}cnicas (CONICET-Argentina).

\end{document}